\documentclass[11pt,twoside]{amsart}
\usepackage{tikz}
\usepackage{tikz-cd}
\usepackage[margin=1.2in]{geometry}
\usepackage{latexsym,amssymb,graphicx}
\usepackage[hidelinks]{hyperref}
\usepackage{mathrsfs}
\usepackage{todonotes}

\newtheorem{theorem}{Theorem}[section]
\newtheorem{lemma}[theorem]{Lemma}
\newtheorem{proposition}[theorem]{Proposition}
\newtheorem{corollary}[theorem]{Corollary}
\newtheorem{MainTheorem}{Theorem}[]

\theoremstyle{definition}
\newtheorem{definition}[theorem]{Definition}

\theoremstyle{remark}
\newtheorem{remark}[theorem]{Remark}

\numberwithin{equation}{section}
\newcommand{\mm}{\hspace{.5mm}}
\newcommand{\mc}{\mathcal}
\newcommand{\ZZ}{\mathbb{Z}}

\newcommand{\spec}{\mathrm{Spec\hskip .5mm }}

\newcommand{\into}{\hookrightarrow}

\newcommand{\mono}{\rightarrowtail}
\newcommand{\epi}{\twoheadrightarrow}
\newcommand{\iso}{\xrightarrow{\sim}}
\DeclareMathOperator*{\colim}{colim}

\DeclareMathOperator*{\hocolim}{hocolim}

\tikzset{
	allign/.style={anchor=north, rotate=90, inner sep=1mm}
}

\address{K Arun Kumar, Department of Mathematics, IISER Mohali, Sector 81, S.A.S Nagar, Manauli PO, Punjab 140306, India.}
\email{karun1993@gmail.com}

\address{Girja S Tripathi, Department of Mathematics, IISER Tirupati, Jangalapalli Village Panguru PO, Tirupati 517619, India.}
\email{girja@labs.iisertirupati.ac.in}

\title{Endomorphisms of Equivariant Algebraic $K$-theory}

\author{K Arun Kumar and Girja S Tripathi}

\begin{document}

\bibliographystyle{alpha}

\begin{abstract} 
We prove that for the action of a finite constant group scheme, equivariant algebraic $K$-theory is represented by a colimit of Grassmannians in the equivariant motivic homotopy category. Using this result we show that the set of endomorphisms of the equivariant motivic space defined by $K_0(G,-)$ coincides with the set of endomorphisms of infinite Grassmannians in the equivariant motivic homotopy category by explicitly computing the equivariant $K$-theory of Grassmannians.
\end{abstract}
\maketitle

\section{Introduction}
 
Equivariant motivic homotopy theory deals with extending tools and results from motivic homotopy theory to the category of schemes equipped with a group scheme action and it has been studied by Krishna, \O stv\ae r, Heller, Voinegeau and Hoyois in \cite{HKO14}, \cite{HVO} and \cite{H20}. There is a long history of research in  generalised cohomology theories in this setting, and one of the important papers in the development of this subject is Thomason's work on algebraic $K$-theory of schemes with group actions in \cite{T87} where he defines equivariant versions of both $G$-theory and $K$-theory. Thomason was able to extend many of the results about algebraic $K$-theory to the equivariant version for a large family of group schemes. Using this, Heller, Krishna, and \O stv\ae r were able to prove (\cite[Corollary~5.2]{HKO14}) an equivariant analogue of Morel and Voevodsky's representability theorem for algebraic $K$-theory \cite[Proposition~4.3.9]{MV99}. We recall some details of this in Section \ref{EquivariantAlgebraicKTheoryGrassmannians}. Morel and Voevodsky were further able to show that algebraic $K$-theory is representable by the infinite Grassmannian presheaf $\ZZ\times Gr$\cite[Theorem ~4.3.13]{MV99}. This result allows us to use the geometric properties of Grassmannians to study algebraic $K$-theory. For instance, J\"oel Riou described the endomorphisms of algebraic $K$-theory using its representability by Grassmannians in \cite{R10}.

For equivariant algebraic $K$-theory, geometric representability was first considered by \O stv\ae r in \cite{O14} in the context of equivariant semi-topological $K$-theory. Recently Hoyois has proved in  \cite{H20} that, for tame group schemes, equivariant algebraic $K$-theory can be represented by the group completion of a disjoint union of Grassmannians.
In Theorems~\ref{thm:simplHtpyEq} and \ref{thm:MT} 
we prove this by using the technique from \cite{ST15} by showing that in the equivariant motivic homotopy category the infinite Grassmannian $\ZZ\times Gr_G$ is a geometric model for equivariant algebraic $K$-theory when $G$ is a constant finite group scheme whose order is invertible over the base. Using this representability we describe endomorphisms of equivariant algebraic $K$-theory analogous to the work done by Riou. It should be noted that Riou has used the cellular decomposition of (non-equivariant) Grassmannians to compute their $K$-groups. In absence of such a cellular decomposition for equivariant Grassmannians we compute their equivariant algebraic $K$-groups in Corollary~\ref{cor:pbtK0} by using the projective bundle theorem.  
This computation is detailed in Section \ref{Sec:K-theory_of_grassmannians}.

Let $G$ be a finite constant group scheme over a noetherian separated base scheme $S$. Denote the category  of $G$-equivariant separated schemes of finite type over $S$ by $Sch^G_S$  and  the  subcategory of smooth $S$-schemes by $Sm^G_S$. Let $M^G(S)$ (respectively $M^G_\bullet(S)$) denote the category of motivic (respectively pointed motivic) spaces over $S$ and $\mc{H}^G(S)$ (respectively $\mc{H}^G_\bullet(S)$) the equivariant motivic (respectively pointed motivic) homotopy category of $G$-schemes over $S$. For the main results of this paper we need the assumption that  the order of the group

be a unit in the ring of global sections $\Gamma(S,\ \mc{O}_S)$. Equivariant algebraic $K$-theory was defined by Thomason by applying Quillen's Q-construction to the exact category of $G$-vector  bundles. 

For a $G$-equivariant $S$-scheme $X$ let us denote the ring of global sections $\Gamma (X,\ \mc{O}_X)$ by $R$. In this paper define equivariant algebraic $K$-theory presheaf $\mathcal{K}\in M^G_\bullet(S)$ by taking the nerve of a functorial version of $S^{-1}S$-construction applied to the additive category $VB(G,\ \spec R)$ of $G$-vector  bundles. Under the assumption that the order of the group is a unit in $\Gamma(S,\ \mc{O}_S)$,
the categories $VB(G,\ \spec R)$ are split exact.

For the ring $R=\Gamma(S,\ \mc{O}_S)$, we have the regular representation $\rho_R$ of $G$. By the regular representation of $G$ over $S$, denoted by $\rho_S$, we will mean $\pi^* \rho_R$, where $\pi:S\rightarrow \spec\Gamma(S,\ \mc{O}_S)$ is the natural map.  This defines Grassmannia $G$-schemes $Gr(m,\mm \rho_S^n)$ over $\spec \Gamma(S,\ \mc{O}_S)$ parametrizing submodules of $\rho_S^n=\rho_S\oplus\cdots\rho_S\ (n\ \mathrm{copies})$ whose quotients are locally free of rank $n\cdot |G|-m$, and the presheaves $Gr(m,\mm \infty)$ and $Gr_G$
via colimits induced by the standard inclusions $\rho_S^n\hookrightarrow \rho_S^{n+1},\ v\mapsto (v,0)$. 
  
In the setting of $\infty$-topos, in \cite[Corollary 2.10 and (5.4)]{H20}, Hoyois proved that for a relatively large family of group schemes the equivariant K-theory space is representable in the unstable motivic homotopy category by the group completion of a disjoint union of infnite Grassmannians .
We use techniques from \cite{ST15} to show that this group completion is in fact the infinite Grassmanian $\ZZ\times Gr_G$ and as an application we have the following theorem.

\begin{MainTheorem}
 Let $S$ be a regular noetherian scheme  of finite Krull dimension and $G\to S$ a constant group $S$-scheme  such that the order of the group is a unit in $\Gamma(S,\ \mc{O}_S)$. The equivariant algebraic $K$-theory is represented by the infinite Grassmannian in the motivic homotopy category, and in particular, for a $G$-scheme $X\in Sm^G_S$,  there are isomorphisms
\[K_n(G, X) \simeq [ X_{+} \wedge S^n, \ \mathbb{Z}\times Gr_G]_{\mc{H}^G_\bullet(S)}\ \ \ \ n\geq 0\] 
in the pointed equivariant motivic homotopy category for $G$-equivariant smooth $S$-schemes. 
\end{MainTheorem} 
For precise details see Theorems \ref{thm:simplHtpyEq} and \ref{thm:MT}. 
For people familiar with the setting it may be possible to prove the above theorem entirely from Hoyois's using the McDuff-Segal Theorem on group completions. We prove it using the approach in \cite{ST15} as we feel that it can be generalised to a broader family of group schemes in the future.

This result allows us to use the geometric properties of $Gr(m,\rho^n)$ to prove equivariant analogues of Riou's results in \cite{R10} on morphisms of the algebraic K-theory presheaf, useful in the study of equivariant characteristic classes.
\begin{MainTheorem}
For any regular noetherian base scheme $S$ of finite Krull dimension and $G\to S$ a constant group $S$-scheme such that the order of the group is a unit in $\Gamma(S,\ \mc{O}_S)$ there is a canonical bijection
\[ [(\ZZ\times Gr_G)^n,\mathbf{R}\Omega^k \ZZ\times Gr_G]_{\mc{H}^G(S)}\iso Hom_{M^G(S)}(K_0(G,\ -)^n,K_k(G,\ -))  \]
for each $n,k\in\mathbb{N}$. In particular,
\[ [\ZZ\times Gr_G,\ZZ\times Gr_G]_{\mc{H}^G(S)}\iso End_{M^G(S)}(K_0(G,\ -)) \]
when $n=1$ and $k=0$.
\end{MainTheorem}
We prove this by computing the ring structure of $K_*(G,\ Gr(m, \rho^n)^r)$ using the projective bundle formula; see Corollary ~\ref{cor:pbtK0}. This bijection for $K_0(G,\ -)$ respects algebraic structures; see Corollary \ref{Algebraic Structures}.
\begin{MainTheorem}
 In the category $\mc{H}^G(S)$ there exists a unique structure of a special $\lambda$-ring with duality on $\ZZ\times Gr_G$ such that the corresponding $\lambda$-ring with duality structure on $K_0(G,X)$ is the standard one for all $X\in Sm^G_S$.
\end{MainTheorem}
  This allows us to lift the Adams operations on the representation ring, uniquely, to higher equivariant $K$-groups. In a forthcoming work we will extend the results above to a broader family of group schemes and use them to study equivariant algebraic cobordism. Specifically we will prove a version of Equivariant Conner Floyd isomorphism \cite{O82} in the motivic setting. \\

\subsubsection*{Notations and conventions.} Throughout this paper, by a scheme we will mean a Noetherian separated scheme of finite Krull dimension. We will call a functor between categories a weak equivalence if the induced map on nerves is a weak equivalence of simplicial sets. To simplify the presentation, we will often use the same notation for a category and its nerve. We will use $\textsl{g}$ to denote the order of the group $G$ in later sections of the paper where it appears very frequently.

\subsubsection*{Acknowledgment}
    This work was carried out while the first author was a postdoctoral fellow at IISER Tirupati. The authors wish to thank the anonymous referee for useful inputs on improving the paper.

\section{Brief recollection of Equivariant motivic homotopy category} 

This section is a quick recap from \cite{HKO14} and \cite{HVO}. For a $S$ of finite Krull dimension and a finite group scheme $G\rightarrow S$, we denote the category of $G$-equivariant motivic spaces over $S$ by $M^G(S)$. It is the category  $\Delta^{op}PShv(Sm_S^G)$  of simplicial presheaves of sets on $Sm^G_S$. There is a global model structure $M^G_{\mathrm{gl}}(S)$ on $M^G(S)$ with global weak equivalences (and global fibrations)  given by schemewise weak equivalence (and schemewise Kan fibrations), and cofibrations defined by the left lifting property.

The $G$-equivariant Nisnevich local model structure $M_{\mathrm{Nis}}^G(S) := L_{\Sigma_{\mathrm{Nis}}^{\mathrm{hp}}}M^G_{\mathrm{gl}}(S)$ is the left Bousfield localization of $M^G_{\mathrm{gl}}(S)$ with respect to the class of morphisms 
\begin{equation}\label{eq:nis-cd-square}\Sigma_{\mathrm{Nis}}^{\mathrm{hp}} : = \{ Q^{\mathrm{hp}} \rightarrow X \}_Q \cup \{ \emptyset \rightarrow h_{\emptyset}\}\end{equation}
where $Q$ runs through all the cartesian squares 
\begin{equation}
\label{eq:nis-cd-square II}
\begin{tikzcd}	
B \arrow[r] \arrow[d] & Y \arrow[d, "p"]
\\
A \arrow[r, hook, "i"] & X 			
\end{tikzcd}
\end{equation}
with $p$ etale and $j$ open immersion such that $ (Y-B)_{\mathrm{red}}  \cong (X-A)_{\mathrm{red}} $, and $Q^{\mathrm{hp}}$ is the homotopy pushout of $Q$ in $M_{\mathrm{gl}}^G(S)$. Here $\emptyset$ is the initial object in the category $\Delta^{op}PShv(Sm_S^G)$ and $h_\emptyset$ is the simplicial presheaf represented by the empty scheme.

The collection of all the cartesian squares (\ref{eq:nis-cd-square II}) forms a $cd$-structure which is complete, regular and bounded. This $cd$-structure induces the equivariant Nisnevich topologies on $Sch_S^G$ and $Sm_S^G$. We will denote the corresponding sites by $[Sch^G_S]_{Nis}$ and $[Sm^G_S]_{Nis}$ respectively. These sites have enough points. 
\begin{theorem}
\label{Points}
Let $G$ be a finite group scheme over $S$. A map $\phi: \mc{F}_1\rightarrow \mc{F}_2$ of sheaves of sets on $[Sch^G_S]_{Nis}$  (respectively $[Sm^G_S]_{Nis}$) is an isomorphism if and only if $\phi_R: \mc{F}_1(\spec R)\rightarrow \mc{F}_2(\spec R)$ is an isomorphism for all (respectively all smooth) semilocal affine $G$-schemes $\spec R$ over $S$ with a single closed orbit for Henselian rings $R$. 
\end{theorem}   
When $S=\spec(k)$ with $k$ a field, this is \cite[Theorem~3.14]{HVO}. By the discussion in the proof of \cite[Proposition 2.17]{HKO14} this generalises to arbitrary $S$. The above theorem is stated in terms of semilocal affine $G$-schemes in \cite{HVO} but it follows from their proof that it suffices to consider semilocal Henselian affine $G$-schemes. 
\begin{remark}
This theorem enables us to recognise equivariant Nisnevich simplicial weak equivalences in the equivariant Nisnevich local model structure.  This is a reason for our choice of the simpler affine version of equivariant $K$-theory presheaf in Definition~\ref{AffineVersionEuivariantKTheory}; see Theorem~\ref{ThomasonEuivariantKTheory}. This fact has also been used in the proof of Theorem \ref{thm:MT}, which in turn relies on Theorem \ref{thm:simplHtpyEq}. 
\end{remark}

The $G$-equivariant motivic model structure $M^G_{\mathrm{mot}}(S) := L_{\Sigma_{\mathbb{A}^1}}M^G_{\mathrm{Nis}}(S)$ is the left Bousfield localizaiton of $M^G_{\mathrm{Nis}}(S)$ with respect to the class of all the projections
\[ \Sigma_{\mathbb{A}^1} := \{ X \times \mathbb{A}^1 \rightarrow X: X \in Sm_S^G \}\]
where $\mathbb{A}^1$ has trivial $G$-action.
The $G$-equivariant motivic homotopy category $\mc{H}^G(S)$ is the homotopy category associated to the equivariant motivic model structure $M^G_{\mathrm{mot}}(S)$. The pointed $G$-equivariant motivic homotopy category $\mc{H}^G_\bullet(S)$ is obtained by a similar procedure from the category of pointed simplicial presheaves $M^G_\bullet(S)=\Delta^{op}PShv_\bullet(Sm_S^G)$. 

\section{ Equivariant algebraic \texorpdfstring{$K-$}{K-}theory and Grassmannians }
\label{EquivariantAlgebraicKTheoryGrassmannians} 
The equivariant algebraic $K$-theory was defined by Thomason by considering equivariant vector bundles on schemes with an action of a group scheme. 
For an action $\theta$ of a group scheme $G$ on a scheme $X\in Sch^G_S$, a $G$-module on $X$ is a quasi-coherent $\mathcal{O}_X$-module $E$ on $X$ together with an isomorphism of $\mathcal{O}_{G\times X}$-modules 
$$\theta^*E\ \xrightarrow{\simeq}\ p_2^*E $$
on $G\times X$
satisfying compatibility of pullbacks of quasi-coherent modules with the action. See \cite[(1.2)]{T87}. A $G$-vector bundle on a $G$-scheme $X$ is $G$-module on $X$ that is locally free as an $\mathcal{O}_X$-module which is locally of finite rank.  For affine $G$-schemes we will use the term $G$-equivariant finitely generated projective module as well for $G$-vector bundle. 

For a scheme with $G$-action $X\in Sm^G_S$  the category $VB(G,\mm X)$ of $G$-vector bundles over $X$  is an exact category. The $i^{th}$ equivariant algebraic $K$-group $K_i(G,\mm X)$ of a $G$-scheme $X$ is defined as the $(i+1)^{\mathrm{st}}$-homotopy group of the nerve of Quillen's $Q$-construction applied to the exact category $VB(G,\mm X)$. Thus,
\[ K_i(G,\mm X) = \pi_{i+1}(\mathscr{N}Q(VB(G,\mm X)), 0),\ \ \ \ i\geq 0\]
where $0$ denotes the class of the zero $G$-vector bundle. For $X=\spec R$ we will also use $K_i(G,\mm R)$ to denote the groups  $K_i(G,\mm X)$. Alternatively we can use the simplicial loop space $(\Omega\mathscr{N}Q(VB(G,\mm X)),0)$ to get the same index on the groups. In Thomason's paper \cite{T87} he uses Waldhausen's $S_\bullet$-construction to define  the equivariant $K$-groups but these give the same homotopy groups for any exact category. The identification 
\[X\mapsto \mc{K}^G(X)=\Omega \mathscr{N}Q(VB(G,\mm X)),\ \forall X\in Sm^G_S \] 
 can be turned into a presheaf by the standard rectification procedure used in algebraic $K$-theory as described for example in \cite[IV.10.5]{W13}.
   We call this the equivariant algebraic $K$-theory presheaf $\mc{K}^G \in M^G(S)$.
 We will now show that for an affine $G$-equivariant scheme $X$ the category $VB(G,\mm X)$ is in fact a split exact category and hence we can use the $S^{-1}S$-construction to define equivariant algebraic $K$-theory for affine $G$-schemes.
 We have the following analogue of Morel and Voevodsky's representability theorem.
 \begin{theorem}\label{thm:first_representability}
 	Let $S$ be a regular noetherian scheme of finite Krull dimension and $G$ a finite group. Then, there are isomorphisms
 	\[K_n(G, X) \simeq [ X_{+} \wedge S^n, \mc{K}^G]_{\mc{H}_\bullet^G(S)} \]
 	for all  $n\geq 0$ and all $X\in Sm^G_S$.
 \end{theorem}
This follows from \cite[Corollary~5.2]{HKO14} and the fact that the constant group scheme $G\to S$ satisfies the resolution property for every $X\in Sm^G_S$ when $S$ is as above \cite[Lemma~2.10]{T87b}. Our goal in this section is to show that $\mc{K}^G$ can be replaced by a directed colimit of Grassmannians as in the non-equivariant case. 
\begin{lemma}
\label{equivariant splitting}
For an affine $G$-scheme $X$, under our assumption in the paper that the order of the group is invertible in $\Gamma(S,\mm \mc{O}_S)$, every short exact sequence of $G$-vector  bundles splits equivariantly, and hence, the category $VB(G,\mm X)$ is split exact.
\end{lemma}
\begin{proof} Given a short exact sequence $0\rightarrow \mathcal{E}\xrightarrow{i} \mathcal{L}\xrightarrow{\phi} \mathcal{M}\rightarrow 0$ in $VB(G,\mm X)$
choose a (non-equivariant) splitting $s:\mathcal{M}\rightarrow \mathcal{L}$ and observe that the morphism 
$$\bar{s}:\mathcal{M}\rightarrow \mathcal{L}, \ \ \ a\mapsto \frac{1}{|G|}\sum_{g\in G}g\cdot s(g^{-1}\cdot a)\ \ \ \ \mathrm{for}\  a\in \mathcal{M}(X)$$
of vector bundles is a $G$-equivariant splitting of $\phi$.  
\end{proof}
Let us denote by $\mathscr{S}_R$ the category having the same objects as the category $VB(G,\mm \spec R)$ but morhpisms being all the isomorphisms in  $VB(G,\mm \spec R)$. The category $\mathscr{S}_R$ is a symmetric monoidal category with respect to direct sum and we can form the category $\mathscr{S}_R^{-1}\mathscr{S}_R$. Recall that objects in the category $\mathscr{S}_R^{-1}\mathscr{S}_R$ are pairs $(A_0,\mm A_1)$ of objects in the category $\mathscr{S}_R$; and, a morphism $(A_0,\mm A_1)\rightarrow (B_0,\mm B_1)$ consists of an equivalence class of triples $(C, a_0,\mm a_1)$ where $C\in \mathscr{S}_R $ and $a_i: A_i\oplus C\rightarrow B_i,\ \ i=0,\mm 1$, are isomorphisms in $\mathscr{S}_R$. In what follows we will describe a small category 
$\mathscr{S}_R^+$ equivalent to the group completion 
$\mathscr{S}_R^{-1}\mathscr{S}_R$. 

\begin{definition}[Regular representation of $G$ over a scheme $S$]
\label{RegRepresentation/Scheme} For an equivariant affine $G$-scheme $(\spec R,\tau)$, let $\rho_{_{R_\tau}}$ (or simply $\rho_{_R}$) be the skew group-ring $R_\tau G$ viewed as the twisted regular representation of $G$ over $R$. The $G$-action on $\rho_{_{R_\tau}}$ is given by \[g'\cdot(\Sigma_{g\in G}x_g\cdot e_g)= \Sigma_{g\in G} (g'\cdot x_g)\cdot e_{g'g},\] where $\{e_g:g\in G\}$ is the standard basis of $\rho_{_{R_\tau}}$. For a scheme $S$ by the regular representation $\rho_S$ of $G$ over $S$ we will mean $\pi^*\rho_{_R}$, the pullback of $\rho_{_R}$  via the natural map $\pi: S\rightarrow \spec\ \Gamma\mm (S,\mm \mc{O}_S)$, where $R=\Gamma\mm (S,\mm \mc{O}_S)$. 
We have direct sums $\rho_S^n$ and embeddings $\rho_S^n\hookrightarrow \rho_S^{n+1},\ v\mapsto (v,0)$  of regular representations. Let $\rho_S^\infty$ be the increasing union of these embeddings as a $G$-module on $S$.
\end{definition}

\begin{lemma}
\label{Equivalent Category of Summands in Regular Representations}
For an affine $G$-scheme $\spec R$, under our assumption in the paper that the order of the group is invertible in $R$, every $G$-equivariant finitely generated projective $R$-module is a summand in a finite direct sum of regular representations.
\end{lemma}

\begin{proof} 
Given an equivariant finitely generated projective $R$-module $P$ choose an epimorphism $\alpha: R^n\epi P$.  The induced surjective map 
\[ \widetilde{\alpha}:\rho_{_{R_\tau}}^n\epi P\ \ \ :\ \ \  (\sum_{g\in G}x_g\cdot e_g )\mapsto \sum_{g\in G}g\cdot \alpha(x_g)\]
is $G$-equivariant and we have an exact sequence of $G$-equivariant finitely generated projective $R$-modules 
$0\to \mathrm{ker}\mm \widetilde{\alpha}\xrightarrow{\ \ } \rho_{_{R_\tau}}^n \xrightarrow{ \ \widetilde{\alpha}\ } P\to 0.$
Choosing an equivariant splitting by Lemma \ref{equivariant splitting} we see that $P$ is a direct summand of $\rho_{{R_\tau}}^n$.
\end{proof}

Now we will define a version of equivariant algebraic $K$-theory presheaf based on affine schemes and in the paper we will exclusively work with this model of equivariant $K$-theory. Not to overburden our notations, we will denote this presheaf simply by $\mathcal{K}$ without any mention of the group $G$ in the notation.  
\begin{definition}[The presheaf $\mathcal{K}$] 
\label{AffineVersionEuivariantKTheory}
Let $\mathscr{S}_R^+\subset \mathscr{S}_R^{-1}\mathscr{S}_R$ be the full subcategory  of pairs $(A,B)$ where $A \subset \rho_R^\infty\oplus \rho_R^\infty=(\rho_R^\infty)^{2}$ and $B \subset  \rho_R^\infty\oplus \rho_R^\infty\oplus \rho_R^\infty=(\rho_R^\infty)^{ 3}$ 
	are $G$-invariant finitely generated projective submodules of $(\rho_R^\infty)^{2}$ and $(\rho_R^\infty)^{3}$, respectively.
	The category $\mathscr{S}_R^+$ is small, functorial in $R$, and by Lemma \ref{Equivalent Category of Summands in Regular Representations} it is equivalent to $\mathscr{S}_R^{-1}\mathscr{S}_R$. Let $\mathcal{K}\in M^G(S)$ be defined by taking the identification  \[\mathcal{K}(X)= \mathscr{N}(\mathscr{S}_R^+),\ \ \ \ \mathrm{where }\ R=\Gamma(X,\ \mathcal{O}_X) \]
for  $X\in Sm^G_S$. 
In this definition we have used $(\rho_R^\infty)^{2}$ and
$(\rho_R^\infty)^{3}$ so that later we can define maps between different presheaves considered. 
\begin{remark}
In subsection \ref{rem:SSasHocolim} we will give an equivalent homotopy colimit description of $K$-theory presheaf $\mc{K}$ with the same notation 
$\mathscr{S}^+$.
\end{remark} 
\end{definition}
For an {\it affine} $G$-scheme $X=\spec R$, the homotopy groups of $\mathcal{K}(X)$ at the base point $0$ recover the equivariant algebraic K-groups of $X$. 
\begin{lemma}\label{thm:affine_Ktheory}
	Let $X=\spec R$ be an affine $G$-scheme such that the order of the group $G$ is invertible in $R$. Then, there is a zigzag of weak equivalences between $(\mc{K}(X),0)$ and $(\mc{K}^G(X),0)$. In particular\[K_i(G,\mm R)=\pi_i(\mathcal{K}(\spec R), 0),\ \ \ \ i\geq 0.\]
\end{lemma}

\begin{proof}
   By Lemma \ref{equivariant splitting} we have that $VB(G,X)$ is split exact and hence the simplicial set $\mathscr{N}(\mathscr{S}_R^{-1}\mathscr{S}_R)$ has the homotopy type of the  loop space of $\mathscr{N}Q(VB(G,\mm X))$  as proved by Quillen \cite{Q73}. By Lemma \ref{Equivalent Category of Summands in Regular Representations}, the inclusion $\mathscr{S}_R^+\subset \mathscr{S}_R^{-1}\mathscr{S}_R$ is an equivalence of symmetric  monoidal categories. Putting these together we get the desired result. 
	\end{proof} 
\begin{theorem} 
\label{ThomasonEuivariantKTheory}
  	The $G$-equivariant motivic spaces $\mc{K}$ and $\mc{K}^G$ are isomorphic as objects in $\mc{H}^G(S)$. 
\end{theorem}
\begin{proof}
	By Theorem \ref{Points} any $G$-equivariant motivic space $\mc{F}\in M^G(S)$ is weakly equivalent to the $G$-equivariant motivic space $X\mapsto \mc{F}(\Gamma(X,\mathcal{O}_X))$ in the $G$-equivariant Nisnevich local model structure. Choosing a functorial loop space construction, Lemma \ref{thm:affine_Ktheory} gives us a global weak equivalence of simplicial presheaves between $\mathcal{K}$  and the presheaf $X\mapsto \mc{K}^G(\Gamma(X,\mathcal{O}_X))$. Putting these two facts together we get a zigzag of weak equivalences between $\mathcal{K}$ and $\mc{K}^G$ in the $G$-equivariant Nishevich local model structure and hence they are isomorphic in $\mc{H}^G(S)$.
\end{proof}
 This theorem shows us that studying $\mc{K}$ suffices for our result. 

\begin{definition}[Reduced equivariant algebraic \texorpdfstring{$K$}{K}-theory]
\label{Reduced Equivariant KTheory}
	For a ring $R$ consider the full subcategory 
	$$\widetilde{\mathscr{S}_R^+} \subset \mathscr{S}_R^+$$ 
	of objects $(A,B) \in \mathscr{S}_R^+$ such that 
	$A \subset \rho_R^\infty =0\oplus \rho_R^\infty \subset (\rho_R^\infty)^{2}$, 
	$B \subset A \oplus \rho_R^\infty \subset 0 \oplus (\rho_R^\infty)^2 \subset (\rho_R^\infty)^3$ and
	$A,B$ have the same constant rank.
	The reduced equivariant algebraic $K$-theory presheaf $\widetilde{\mc{K}}$ is defined by 
	\[\widetilde{\mathcal{K}}(X)= \mathscr{N}(\widetilde{\mathscr{S}_R^+}),\ \ \ \ \mathrm{where }\ X\in Sm^G_S\ \mathrm{and}\ R=\Gamma(X,\ \mathcal{O}_X). \]
	For an {\it affine} smooth $G$-scheme $X=\spec R$, its reduced equivariant algebraic $K$-groups $\widetilde{K}_i(G,\mm R)$ are the homotopy groups of
	$\widetilde{\mathcal{K}}(X)$ with respect to the base point $0$. Thus, 
	\[\widetilde{K}_i(G,\mm R)=\pi_i(\widetilde{\mathcal{K}}(\spec R), 0)\ \ \ \ i\geq 0.\]
\end{definition}
The equivalence class of any $(A,B)\in \widetilde{\mathscr{S}_R^+}$ will correspond to the element $[A]-[B]\in K_0(G,R)$.
\subsection{Reduced and unreduced equivariant algebraic \texorpdfstring{$K$}-theories} 
Consider the monoid $\mathbb{N}$ of non-negative natural numbers as a discrete symmetric monoidal category with respect to addition.
The ring $R$ is a $G$-invariant summand of $\rho_R$ via 
\[ R\subset \rho_R,\ \ \  x\mapsto\sum_{g\in G}x\cdot e_g ,\ \ \ x\in R, \]
and, it defines a $G$-invariant submodule $R^n\subset \rho^n_R.$ There is a map of presheaves from  the categorical group completion of $\mathbb{N}$ to $\mc{K}$:
\[ \mathbb{N}^{-1}\mathbb{N} \longrightarrow \mc{K}: (n,\mm m)\mapsto (R^n,\mm R^m) \]
where the factor $R^n$ embeds in $\rho^n_R\oplus 0\subset (\rho^\infty_R)^2$ and the factor $R^m$ embeds in $\rho^m_R\oplus 0\oplus 0\subset (\rho^\infty_R)^3.$
The full subcategory $\widetilde{\mathscr{S}_R^+} \subset \mathscr{S}_R^+$
defines a map of presheaves
\begin{equation}
\label{ReducedNonreduced}
\mathbb{N}^{-1}\mathbb{N}\times \widetilde{\mc{K}}\longrightarrow \mc{K}:( (n,\mm m),(A,\mm B))\mapsto (R^n\oplus A,\mm R^m\oplus B). 
\end{equation}
For equivariant schemes the following notion of $G$-connectedness seems more appropriate. 
\begin{definition}[$G$-connected schemes]
\label{G-connected} Given a $G$-scheme $X$, we will call it $G$-connected if the group acts transitively on its connected components.  
\end{definition}
 Every $G$-connected scheme has a Nisnevich cover by $G$-connected affine schemes; and, every $G$-vector  bundle on a $G$-connected scheme has constant rank. These allow us to relate reduced and non-reduced equivariant $K$-theories as follows.
\begin{lemma}\label{lem:reducedK}
The above map in $\mathrm{(\ref{ReducedNonreduced})}$ is an isomorphism in the motivic homotopy category 
\[\mathbb{Z}\times \widetilde{\mc{K}}\xrightarrow{\ \sim\ } \mc{K}.\]
\end{lemma}
\begin{proof}
Similar arguments as in  \cite[Lemma 4.4]{ST15} apply here as well.
\end{proof}

\subsection{Grassmannians indexed over representations}

\label{RegularRepresentationsCategoricalVersion} 
In Definition \ref{RegRepresentation/Scheme} we have introduced $\rho^n$ and $\rho^\infty$. 
For a scheme $X\in Sm^G_S$,
with $R=\Gamma(X,\mm\mc{O}_X)$ the ring of global sections, we have regular representations $\rho^n_R$ and inclusions $\rho^n_R \subset \rho^{n+1}_R,
 v\mapsto (v,0)$. The $G$-module $\rho^\infty_R$ is the increasing union of $G$-equivariant locally free $R$-modules $\rho^{n}_R
\subset\rho^{n+1}_R\subset\cdots,\ n\geq 0$.
The set of representations $\rho^n_R$ with inclusions $\rho^n_R \subset \rho^{n+1}_R, v\mapsto (v,0)$ is isomorphic to the ordered set $\mathbb{N}$. For the base scheme $S$ we will use notations $\rho^n_S$ and $\rho^\infty_S$. When defining various presheaves later in this section indexed over representations, we will use the notation 
 $\varrho$ for emphasis instead of $\mathbb{N}$. 
 
 From now on in the paper we will use $\textsl{g}$ for the order $|G|$ of the group $G$ to keep the presentation less cumbersome.

\begin{definition}[Grassmannians of $G$-subbundles of regular representations] 
\label{GrassmannianSchemePresheaf}
\label{InfiniteGrassmannian}
Given integers $m\geq 0, \ n\geq 1,$ we have the scheme $Gr(m,\mm \rho_S^{n})$ parametrizing the $G$-equivariant submodules of $\rho_S^{n}$ which are locally free of constant rank $m$ and whose quotients are locally free. We call such submodules $G$-subbundles. The underlying scheme of $Gr(m,\mm \rho_S^{n})$ is the usual Grassmannian $Gr(m,\mm n\cdot \textsl{g})$ over $S$. This scheme has a natural $G$-action induced by the $G$-action on $\rho^{n}_S$, and for any $G$-scheme $X$, a morphism $X\to Gr(m,\mm \rho_S^{n})$ is $G$-equivariant if and only if the corresponding submodule is $G$-invariant. The embeddings $\rho_S^n\hookrightarrow \rho_S^{n+1},\ v\mapsto (v,0)$, define a map of Grassmannians $Gr(m,\mm \rho_S^{n})\hookrightarrow Gr(m,\mm \rho_S^{n+1})$, and one version of $G$-equivariant infinite Grassmannians is the ind-scheme
\begin{equation}
\label{infinite Grassmannian (m, infinity)}
Gr(m,\mm \infty)=Gr(m,\mm \rho_S^\infty)=\colim_{n\geq 1}\ [Gr(m,\mm \rho_S^{n})\hookrightarrow Gr(m,\mm \rho_S^{n+1})\hookrightarrow\cdots]. 
\end{equation}
We will view these as presheaves via Yoneda embedding. 
The embeddings $\rho_R^n\hookrightarrow \rho_R^{n+1}=\rho_R\oplus \rho^n_R,\ v\mapsto (0,v)$, induce maps $Gr(m,\mm \rho_S^n)\hookrightarrow Gr(m+\textsl{g},\mm \rho_S\oplus \rho_S^{n+1})$ $:E\mapsto \rho_R\oplus E$ of finite Grassmannians. In the colimit as in (\ref{infinite Grassmannian (m, infinity)}), these induce maps of inifinite Grassmannians $Gr(m,\mm \rho_S^\infty)\hookrightarrow Gr(m+\textsl{g},\mm \rho_S\oplus \rho_S^\infty)$. Another version of infinite Grassmannians is the colimit presheaf
\begin{equation}
\label{Gr_G definition}
Gr_G=\colim\ [Gr(m\textsl{g},\mm \rho_S^\infty)\hookrightarrow Gr(m\textsl{g}+\textsl{g},\mm \rho_S\oplus \rho_S^\infty)\hookrightarrow Gr(m\textsl{g}+2\textsl{g},\mm \rho^2_S\oplus \rho_S^\infty)\hookrightarrow\cdots].
\end{equation}
\end{definition}
 
\begin{remark}
\label{Gr(F)} 
By taking embeddings $\rho^n_S\oplus \rho_S^k\hookrightarrow \rho^n_S\oplus \rho_S^{k+1}\hookrightarrow \rho^n_S\oplus \rho_S^{k+2}\cdots$, we can define 
$Gr(d,\ \rho^n_S\oplus \rho^\infty_S)$ as in (\ref{infinite Grassmannian (m, infinity)}). 
The colimit 
\begin{equation}
\label{pointed Gr_G definition}
\colim\ [Gr(0,\mm \rho_S^\infty)\hookrightarrow Gr(\textsl{g},\mm \rho_S\oplus \rho_S^\infty)\hookrightarrow Gr(2\textsl{g},\mm \rho^2_S\oplus \rho_S^\infty)\hookrightarrow\cdots].
\end{equation}
obtained by taking $m=0$ in (\ref{Gr_G definition}) is isomorphic to $Gr_G$ and as a motivic space this colimit can be pointed by the subspace $0\in Gr(0,\rho^\infty_S)$. This will be a version of infinite Grassmannian as a pointed motivic space.
\end{remark}

\subsection{Grassmannian as equivariant \texorpdfstring{$K$}{K}-theory in motivic homotopy category} 
We will now define presheaves $\mc{S}(\rho^n\oplus\rho^\infty)$ and $\mc{S}(d,\ \rho^n\oplus\rho^\infty), \ d\geq 0$, of categories on $Sm^G_S$. These provide a  version of group completion for equivariant algebraic $K$-theory
and enable us to see $Gr_G$ as a geometric model of equivariant algebraic $K$-theory. 

\begin{definition}[Presheaves of categories $\mc{S}(\rho^n\oplus\rho^\infty)$ and $\mc{S}(d,\ \rho^n\oplus\rho^\infty)$]
\label{CategoricalGroupCompletionPresheaves} 
\label{S(F)}
Let $X$ be an equivariant scheme in $Sm^G_S$ with $R=\Gamma (X,\mm \mc{O}_X)$. 
Define the category $\mc{S}(\rho^n\oplus\rho^\infty)(X)$ as follows. Its objects
are $G$-equivariant locally free submodules $V\subset \rho^n_R\oplus\rho^\infty_R$ of constant finite rank whose quotients are locally free; and,
given objects $V, \ W$ in $\mc{S}(\rho^n\oplus\rho^\infty)(X)$,  morphisms $V\rightarrow W$ are $G$-equivariant isomorphisms of submodules $V$ and $W$, not required to be compatible with embeddings into $\rho^n_R\oplus\rho^\infty_R$.
Via pullback of sheaves of modules the assignment 
$X\mapsto \mc{S}(\rho^n\oplus\rho^\infty)(X)$ defines a presheaf of categories $\mc{S}(\rho^n\oplus\rho^\infty)$.
For a non-negative integer $d$, denote by $\mc{S}(d,\ \rho^n\oplus\rho^\infty)$, the subpresheaf of 
$\mc{S}(\rho^n\oplus\rho^\infty)$  defined by full subcategories consisting of submodules of constant rank $d$.	
\end{definition}

\begin{definition}[Group completion of presheaf of categories $\mc{S}^+$]
\label{mathcalS+}
We have the functor 
\begin{equation}
\label{S^+}
\mc{S}(-, \ - \oplus \rho^{\infty}_S):\varrho\rightarrow \mathrm{Cat}\ \ \ :\ \ \  \rho^n\mapsto \mc{S}(n\cdot \textsl{g}, \ \rho^n_S\oplus \rho^{\infty}_S)
\end{equation}
from $\varrho$ into the category $\mathrm{Cat}$ of small categories: Here the transition map  
for $\rho^n_S\subset \rho^{n+k}_S=\rho^n_S\oplus \rho^k_S\ , v\mapsto (v,0) $ 
is defined by $V\mapsto V\oplus \rho^k_S$ on objects and by $\alpha\mapsto \alpha \oplus 1_{\rho^k_S}$ on morphisms.
The group completion presheaf $\mc{S}^+$ is defined as the following colimit
$$\mc{S}^+ = \colim_{\rho^n_S\in  \varrho_S} \mc{S}(n\cdot \textsl{g},\ \rho^n_S\oplus \rho^{\infty}_S).$$
\end{definition}
The reason for calling these presheaves group completion are the weak equivalences in (\ref{eqn:SplusSplusiso}) and in the Theorem \ref{thm:simplHtpyEq}.

\subsection{Maps from Grassmannian into group completion presheaf \texorpdfstring{$\mc{S}^+$}{S+}.} 
The set of objects in $\mc{S}(\rho^n_S\oplus \rho^{\infty}_S)$ is given by the disjoint union $\coprod_{d\geq 0}\mc{S}(d,\ \rho^n_S\oplus \rho^{\infty}_S)$. The set of objects in $\mc{S}(n\cdot 
\textsl{g},\ \rho^n_S\oplus \rho^{\infty}_S)$ defines the representable presheaf $Gr(n\cdot \textsl{g}, \ \rho^n_S\oplus \rho^{\infty}_S)$. 
The presheaves of sets $Gr(n\cdot \textsl{g}, \ \rho^n_S\oplus \rho^{\infty}_S)$, thought of as presheaves of discrete categories, define the functor 
\begin{equation}
\label{Gr_G functor for colim and hocolim}
Gr(-, \ -\oplus \rho^{\infty}_S):\varrho\rightarrow \mathrm{Cat}\ \ \ :\ \ \  \rho^n\mapsto Gr(n\cdot \textsl{g}, \ \rho^n_S\oplus \rho^{\infty}_S)
\end{equation}
from the index category $\varrho$ to the category $\mathrm{Cat}$ of small categories. Then
$$Gr_G = \colim_{\rho^n_S\in  \varrho_S} Gr(n\cdot \textsl{g},\ \rho^n_S\oplus \rho^{\infty}_S).$$
The inclusion of objects in a category defines a natural transformation $Gr(-, \ -\oplus \rho^{\infty}_S)\rightarrow 
\mc{S}(-, \ - \oplus \rho^{\infty}_S)$, and we have the induced functor of colimit
categories, and hence a map of simplicial presheaves 
\begin{equation}
\label{mathcalGr2mathcalS+}
Gr_G \to  \mc{S}^+
\end{equation}
from the representable presheaf defined by infinite Grassmannians $Gr_G$ into the group completion
$ \mc{S}^+$; see Remark \ref{Gr(F)}.

\subsection{Equivariant algebraic \texorpdfstring{$K$}{K}-theory as a homotopy colimit}
\label{rem:SSasHocolim}
We will now explain that instead of colimit in Definition \ref{mathcalS+} if we take Thomason's Grothendieck construction or lax colimit, then we get the functorial version of reduced equivariant algebraic $K$-theory defined in \ref{Reduced Equivariant KTheory} via equivalence of categories in (\ref{eqn:reduced K-Theory as Grothendieck construction}). In \cite[Definition 7]{ST15}, this construction was called homotopy colimit, because homotopy colimit of the diagram of nerve simplicial sets indexed over $\varrho$ is weakly equivalent to the nerve of lax colimit, but we should mention that usage of term homotopy colimit here does conflict with the use of the term homotopy colimit in quasi-categories. We will still use the term homotopy colimit in this paper.
For the convenience of reader we recall this definition from \cite[Definition 7]{ST15}.

\begin{definition}[Homotopy colimits]
\label{dfn:hocolim}
Let $\mc{F}: \mathcal{C}\rightarrow \mathrm{Cat}$ be a functor from a category $\mathcal{C}$ to the category $\mathrm{Cat}$ of small
categories. 
The homotopy colimit 
$$\hocolim_{\mathcal{C}}\mc{F}$$
is the category whose objects are pairs
$(X,A)$ with $X$ an object of $\mathcal{C}$ and $A$ an object of $\mc{F}(X)$; and, a map from $(X,A)$ to $(Y,B)$ is a pair $(x,a)$, where $x:X \rightarrow Y$ is a map in
$\mathcal{C}$ and $a: F(x)A \rightarrow B$ is a map in $\mc{F}(Y)$.
Composition $(y,b)\circ (x,a)$ of $(x,a):(X,A) \rightarrow 
(Y,B)$ and $(y,b):(Y,B) \rightarrow (Z,C)$ is the map $(y\circ x, b\circ \mc{F}(y)a)$.
\end{definition}
Let us denote the homotopy colimit of the functor in (\ref{S^+})
by $\widetilde{\mathscr{S}^+}$, that is,
\begin{equation}
\label{hocolim S^+}\widetilde{\mathscr{S}^+}= \hocolim_{\varrho} \mc{S}(-,\ -\oplus \rho^{\infty}_S).
\end{equation}
For a scheme $X\in Sm^G_S$ with $ R=\Gamma(X,\ \mathcal{O}_X)$, 
the category $\widetilde{\mathscr{S}^+}(\spec R)$ has objects the pairs $(\rho^n_R, A)$, where $A\subset \rho^n_R\oplus\rho^{\infty}_R$ is
an equivariant submodule of rank $n\cdot \textsl{g}$. Consider the fully faithful functor 
$$\widetilde{\mathscr{S}^+}(\spec R) \rightarrow \widetilde{\mathscr{S}_R^+}\ \ : \ \ (\rho^n_R,A)\mapsto (\rho^n_R,A),$$
see Definition \ref{Reduced Equivariant KTheory}. It is essentially surjective
since an object $(A,B)$ in $\widetilde{\mathscr{S}_R^+}$ is isomorphic to an object of the form $(\rho^n_R,B\oplus A')$ where we choose an equivariant submodule $A'$ such that $A\oplus A'\simeq \rho^n_R$ for some positive integer $n$.
Thus, 
we have an equivalence of categories
\begin{equation}
\label{eqn:reduced K-Theory as Grothendieck construction}
\widetilde{\mathscr{S}_R^+} \simeq \widetilde{\mathscr{S}^+}(\spec R).
\end{equation}
This is a homotopy colimit description of reduced equivariant $K$-theory presheaf 
$$\widetilde{\mathcal{K}}(X)= \mathscr{N}(\widetilde{\mathscr{S}_R^+})=
\mathscr{N}(\hocolim_{\varrho} \mc{S}(-,\ -\oplus \rho^{\infty}_S)(\spec R)).$$
\begin{remark}
Taking homotopy colimit of the functor given by the presheaf of  categories $\mc{S}(\rho^n\oplus\rho^\infty)$ we get an equivalent description of the nonreduced $K$-theory $\mc{K}$ defined in \ref{AffineVersionEuivariantKTheory}.
\end{remark}

\subsection{Grassmannian and Group Completion as homotopy colimits}
	\label{dfn:GrRtoGW}

Since the indexing category $\varrho$ is directed, the natural map from homotopy colimit to colimit  
	\begin{equation}\label{eqn:SplusSplusiso}
	\widetilde{\mathscr{S}^+ }\stackrel{\sim}{\to} \mc{S}^+ \end{equation}
induces a weak equivalences of simplicial sets for every $G$-scheme, see \cite[Lemma 13]{ST15}. 
Given an affine scheme $\spec(R)$ equipped with a $G$-action,
let $\Delta R$ be the $G$-equivariant simplicial ring $n\mapsto \Delta^nR = R[T_0,...,T_n]/\langle T_0 + \cdots + T_n - 1\rangle$ where the $G$-action on $\Delta^nR$ corresponds to the trivial $G$-vector bundle action on $\spec(\Delta^nR)\cong \mathbb{A}^n_R$.

\begin{theorem}
\label{thm:simplHtpyEq}
Let $R$ be a commutative $G$-connected $G$-equivariant regular noetherian ring in which 
the order of the group $G$ is a unit. The maps in (\ref{mathcalGr2mathcalS+}) and (\ref{eqn:SplusSplusiso}) induce weak equivalences of simplicial sets
$$Gr_G(\Delta R) \stackrel{\sim}{\rightarrow}\mc{S}^+(\Delta R)\xleftarrow{\sim} 
\widetilde{\mathscr{S}^+}(\Delta R) = \widetilde{\mathcal{K}}(\Delta R).$$ 
Thus we have an isomorphism 
\[
Gr_G \xrightarrow{\ \sim\ } \widetilde{\mathcal{K}}
\]
in $\mathcal{H}^G_\bullet(S)$. 
\end{theorem}

\begin{proof}
To prove this theorem the general machinery developed in \cite[Sections 6 and 7]{ST15} for the proof of representability of Grothendieck-Witt theory by orthogonal Grassmannian can be applied. 
We wish to prove that the map $Gr_G(\Delta R)\rightarrow 
\mc{S}^+(\Delta R)$ is a weak equivalence of simplicial sets.  We have surjections of presheaves  
$$Gr(n\cdot \textsl{g},
	\mm \rho^n_S\oplus\rho^{\infty}_S)\rightarrow \widetilde{K}_0(G,\mm R)\ \ \mathrm{and} \ \ \mc{S}(n\cdot \textsl{g},
	\mm \rho^n_S\oplus\rho^{\infty}_S)\rightarrow \widetilde{K}_0(G,\mm R)$$
    defined by sending  any $A\subset \rho^n_S\oplus\rho^{\infty}_S$ to $[A]-[\rho_S^n]$. For any such element $[A]-n[\rho_S]\in \tilde{K}_0(G,R)$ let $Gr_{G,A}\subset Gr_G$ and $\mc{S}^+_{A}\subset \mc{S}^+$ denote the corresponding fibres with respect to the surjection. These are given, as in \ref{InfiniteGrassmannian} and \ref{mathcalS+}, via induced colimits
$$Gr_{G,A}=\colim_{\rho^n_S\in \varrho_S} Gr(A\oplus\rho^n_S,
	\mm \rho^{m+n}_S\oplus\rho^{\infty}_S) \ \ \ \ \ \ \ \mathrm{and}\ \ \ \ \ \ 
	\mc{S}^+_{A} =\colim_{\rho^n_S\in \varrho_S} \mc{S}(A\oplus\rho^n_S, \rho^{m+n}_S\oplus \rho^{\infty}_S)
$$
where $Gr(A\oplus\rho^n_S,
	\mm \rho^{m+n}_S\oplus\rho^{\infty}_S)$ and $\mc{S}(A\oplus\rho^n_S, \rho^{m+n}_S\oplus \rho^{\infty}_S)$ are defined by taking sub-objects in $\rho^{m+n}_S\oplus \rho^{\infty}_S$ isomorphic to $A\oplus \rho^m_S$. Note that while our construction depends on the choice of presentation for an element of $\tilde{K}_0(G,R)$, the colimits will be the same.
 Therefore we get cartesian diagrams of simplicial sets  
\begin{equation}
\begin{tikzcd}	
Gr_{G,A}(\Delta R) \arrow[r] \ar[d]& Gr_{G}(\Delta R)  \ar[d] 
& & &  \mc{S}^+_{A}(\Delta R)\ar[d] \ar[r] & \mc{S}^+(\Delta R) \ar[d]\\ 
pt\ar[r] &\widetilde{K}_0(G,\Delta R) & & & pt \ar[r] &\widetilde{K}_0(G, \Delta R)
\end{tikzcd}
\end{equation}
for each $A$.
In these diagrams the lower horizontal map is a fibration since for regular rings equivariant $K$-theory is homotopy invariant, and hence for connected regular rings we have $\widetilde{K}_0(G,\mm \Delta R)\iso \widetilde{K}_0(G,\mm R)$,  is a constant simplicial set. 
Thus, both these diagrams are homotopy cartesian diagrams of simplicial sets
with fibrations $Gr_{G}(\Delta R)\rightarrow \widetilde{K}_0(G,\Delta R)$ and $\mc{S}^+(\Delta R)\rightarrow \widetilde{K}_0(G,\Delta R)$.

We now have a family of diagrams 
\begin{equation}
\label{MainTheoremFibrationDiagram}
\begin{tikzcd}	
Gr_{G,A}(\Delta R) \arrow[r] \ar[d]& Gr_{G}(\Delta R) \arrow[r] \ar[d] 
& \widetilde{K}_0(G,\mm \Delta R)	\ar[d,"1"]\\
\mc{S}^+_{A}(\Delta R) \ar[r] & \mc{S}^+(\Delta R)  \ar[r] &
\widetilde{K}_0(G,\mm \Delta R)
\end{tikzcd}
\end{equation}
of homotopy fibrations of simplicial sets.  
In Proposition \ref{Connected Component Identification} below we show that the left vertical map in the above diagram is a weak equivalence of connected simplicial sets. 

Thus, in the above homotopy fibration diagram we have homotopy equivalence between the fibers at each base-point and we have bijections,
$\pi_0(\mc{S}^+(\Delta R))= \widetilde{K_0}(G,\mm R)\cong \pi_0(Gr_{G}(\Delta R))$. The induced long exact sequence of homotopy groups will prove that the map $Gr_G(\Delta R)\rightarrow \mathcal{S}^+(\Delta R)$ is a weak equivalence provided all the terms are abelian groups and the bijection $\pi_0(Gr_{G}(\Delta R))\cong \widetilde{K_0}(G,\mm R)$ is a group homomorphism.
For this one can follow the method from \cite[Section 7]{ST15} to show that the fibration diagram is a diagram of $E_\infty$-spaces. The $E_\infty$-operad $\mathscr{E}$ is defined as follows: The $n$-th space $\mathscr{E}(n)$ of the operad is the space of $G$-equivariant embeddings of $(\rho^\infty)^n$
into $\rho^\infty$. Thus, for a commutative ring $R$, 
$$\mathscr{E}(n)(R)=\lim_{\rho^k_R\in \rho^\infty_R}\mathrm{St}((\rho^k_R)^{\oplus n}, \rho^\infty _R)$$
where $\mathrm{St}((\rho^k_R)^{\oplus n}, \rho^\infty _R)$ denotes the set of 
$G$-equivariant embeddings; see (\ref{Contractibility of Stiefel presheaf}). The transition maps for $\rho^k\subset \rho^{k+l},\ \ v\mapsto (v,0)$ are given by restrictions $\mathrm{St}((\rho^{k+l}_R)^{\oplus n}, \rho^\infty _R)
\rightarrow \mathrm{St}((\rho^k_R)^{\oplus n}, \rho^\infty _R)$.
The rest of the details follow exactly as in \cite[Section 7]{ST15}.
\end{proof}	

\begin{proposition}
\label{Connected Component Identification}
The left vertical map $Gr_{G,A}(\Delta R)\rightarrow \mc{S}^+_{A}(\Delta R)$ in the diagram $\mathrm{(\ref{MainTheoremFibrationDiagram})}$
is a  weak equivalence of connected simplicial sets.
\end{proposition}
\begin{proof} 
For simplicity of notation we deal with the case of $A=0$ below the same argument works for arbitrary $A$ as by Lemma \ref{Contractibility of Stiefel presheaf}. In the next paragraph we construct contractible simplicial sets $\mathrm{St}(\rho^n,\ \rho^n\oplus \rho^{\infty})(\Delta R)$ and $B\mathcal{E}(\rho^n,\ \rho^n\oplus \rho^{\infty})(\Delta R)$ with free actions of the simplicial group $\mathrm{Aut}(\rho^n_{\Delta R})$, and an equivariant map such that the induced map on quotients is the map $Gr(\rho^n,
\mm \rho^n\oplus\rho^{\infty})(\Delta R)\rightarrow \mc{S}(\rho^n,
\mm \rho^n\oplus\rho^{\infty})(\Delta R) $ 
given by inclusion of zero simplicies. This implies these simplicial sets are connected and the map is a weak equivalence by \cite[Proposition 8]{ST15}.  
Being a directed colimit of these weak equivalence, the map $Gr_{G,[0]}(\Delta R)\rightarrow \mc{S}^+_{[0]}(\Delta R)$ in the proposition is a weak equivalence.

Let $\mathrm{St}(\rho^n,\ \rho^n\oplus \rho^{\infty})$ denote the presheaf of sets defined by taking $G$-equivariant embeddings $\rho^n\rightarrow \rho^n\oplus \rho^{\infty}$. The group $\mathrm{Aut}(\rho^n)$ of automorphisms of $\rho^n$ acts freely on the right on $\mathrm{St}(\rho^n,\ \rho^n\oplus \rho^{\infty})$ and the quotient
$\mathrm{St}(\rho^n,\ \rho^n\oplus \rho^{\infty})/\mathrm{Aut}(\rho^n)$ 
is isomorphic to $Gr(\rho^n,\rho^n\oplus \rho^{\infty})$.
Let
$\mathcal{E}(\rho^n,\ \rho^n\oplus \rho^{\infty})$ be the presheaf of categories having objects $\mathrm{St}(\rho^n,\ \rho^n\oplus \rho^{\infty})$; and, given embeddings $a,b:\rho^n\rightarrow \rho^n\oplus \rho^{\infty}$, a morphism  $a\rightarrow b$ is the unique isomorphism $\mathrm{Im}(a)\rightarrow \mathrm{Im}(b)$ such that the diagram 
\begin{center}
\begin{tikzcd}	
\rho^n \ar[r,"a"] \arrow[rd, "b"]& \mathrm{Im}(a) \ar[d] \\
\ & \mathrm{Im}(b)	
\end{tikzcd}
\end{center}
commutes.
The group $\mathrm{Aut}(\rho^n)$ of automorphisms of $\rho^n$ acts freely on the right on classifying space $B\mathcal{E}(\rho^n,\ \rho^n\oplus \rho^{\infty})$ and the the quotient 
$(B\mathcal{E}(\rho^n,\ \rho^n\oplus \rho^{\infty}))/\mathrm{Aut}(\rho^n)$
isomorphic to the classifying space $B\mc{S}(\rho^n,\rho^n\oplus \rho^{\infty})$, here by classifying space of a category we mean its nerve. 
By taking quotients, the map given by inclusion of objects $\mathrm{St}(\rho^n,\ \rho^n\oplus \rho^{\infty})(\Delta R)\rightarrow \mathcal{E}(\rho^n,\ \rho^n\oplus \rho^{\infty})(\Delta R)$ 
induces the map $Gr(\rho^n,
\mm \rho^n\oplus\rho^{\infty})(\Delta R)\rightarrow \mc{S}(\rho^n,
\mm \rho^n\oplus\rho^{\infty})(\Delta R)$ as in (\ref{mathcalGr2mathcalS+}).

The simplicial set $\mathrm{St}(\rho^n,\ \rho^n\oplus \rho^{\infty})(\Delta R)$ is contractible as follows. 
The group $\mathrm{Aut}((\rho^n\oplus\rho^\infty)_{\Delta R})$ acts transitively on the left on $\mathrm{St}(\rho^n,\ \rho^n\oplus \rho^{\infty})(\Delta R)$. For this action the stabiliser of the embedding $a:\rho^n_{\Delta R}\rightarrow (\rho^n\oplus \rho^{\infty})_{\Delta R},\ v\mapsto (v,0)$ is the group $\mathrm{Aut}(\rho^\infty_{\Delta R})$, and hence $\mathrm{St}(\rho^n,\ \rho^n\oplus \rho^{\infty})(\Delta R)$ is isomorphic to $\mathrm{Aut}(\rho^\infty_{\Delta R})\backslash \mathrm{Aut}((\rho^n\oplus\rho^\infty)_{\Delta R})$. Now the contractibility of 
$\mathrm{St}(\rho^n,\ \rho^n\oplus \rho^{\infty})(\Delta R)$ follows from the homotopy equivalence in 
Lemma \ref{Contractibility of Stiefel presheaf} applied to $\rho^n$ 
by the analogous  version of \cite[Proposition 2]{ST15}.
The presheaf $\mathcal{E}(\rho^n,\ \rho^n\oplus \rho^{\infty})$  is also contractible since in this category every object is initial.  Thus, 
the map $Gr(\rho^n,
\mm \rho^n\oplus\rho^{\infty})(\Delta R)\rightarrow \mc{S}(\rho^n,
\mm \rho^n\oplus\rho^{\infty})(\Delta R) $ 
is a weak equivalence by \cite[Proposition 8]{ST15}, and the induced map of this proposition $Gr_{G,[0]}(\Delta R)\rightarrow \mc{S}^+_{[0]}(\Delta R)$ on directed colimits is a weak equivalence.     
\end{proof}
\begin{lemma} 
\label{Contractibility of Stiefel presheaf}
Given a $G$-vector  bundle $V$ on an equivariant affine scheme $\spec R$, the inclusion
$\rho^\infty_{\Delta R}\subset V\oplus\rho^\infty_{\Delta R}$ induces a homotopy equivalence of simplicial groups
$$\mathrm{Aut}(\rho^\infty_{\Delta R})\subset 
\mathrm{Aut}(V\oplus\rho^\infty _{\Delta R})\ \ \ : \ \ \ \sigma\mapsto 1_V\oplus \sigma.$$
\end{lemma}
\begin{proof} The proof follows exactly as in \cite[Lemma 6]{ST15} with the choice of $g=h\oplus h^{-1}\in \mathrm{Aut}(\rho^{2n+2}_{\Delta R})$ where $h\in \mathrm{Aut}(\rho^{n+1}_{\Delta R})$ is the isometry given by the matrix
$\left( \begin{smallmatrix} 0 & 1 \\ 1_{n} & 0\end{smallmatrix}\right)$.  
We use the matrix identity
$$
\left(\begin{smallmatrix}h&0\\ 0&h^{-1}\end{smallmatrix}\right)=
\left(\begin{smallmatrix}1&h\\ 0&1\end{smallmatrix}\right)
\left(\begin{smallmatrix}1&0\\ -h^{-1}&1\end{smallmatrix}\right)
\left(\begin{smallmatrix}1&h\\ 0&1\end{smallmatrix}\right)
\left(\begin{smallmatrix}1&-1\\ 0&1\end{smallmatrix}\right)
\left(\begin{smallmatrix}1&0\\ 1&1\end{smallmatrix}\right)
\left(\begin{smallmatrix}1&-1\\ 0&1\end{smallmatrix}\right),
$$
and $\mathbb{A}^1$-contractibility of elementary matrices.
\end{proof}

Putting together Theorem \ref{thm:simplHtpyEq}  and Lemma \ref{lem:reducedK} we get the main theorem.
\begin{theorem}
	\label{thm:MT} 
	Let $S$ be a regular noetherian scheme of finite Krull dimension such that the order of the group is a unit in $\Gamma(S,\ \mc{O}_S)$. Let $X\in Sm^G_S$ be a $G$-equivariant smooth scheme over $S$. Then,  there are isomorphisms
	\[K_n(G, X) \simeq [ X_{+} \wedge S^n, \ZZ\times Gr_G]_{\mc{H}_\bullet^G(S)} \] for all $n$.
\end{theorem}
Here we have written out the conditions on the scheme $S$ for clarity.
\section{Operations in equivariant algebraic \texorpdfstring{$K$}{K}-theory}

In this and following sections the base scheme $S$ will be assumed to be regular and $G\to S$ will be a constant finite group scheme whose order $|G|=\textsl{g}$ is a unit in $\Gamma(S,\mathcal{O}_S)$. When the required results are true for a broader class of schemes it will be explicitly mentioned. For each $k\geq 0$, let $\mathbf{R}\Omega^k:\mc{H}_\bullet^G(S)\to \mc{H}_\bullet^G(S)$ be the derived $k^{th}$ simplicial loop space functor. The rest of this paper will be dedicated to prove the following theorem.
	\begin{theorem}\label{thm:main}
		 There is a canonical bijection
		\[ [(\ZZ\times Gr_G)^n,\mathbf{R}\Omega^k \ZZ\times Gr_G]_{\mc{H}^G(S)}\iso Hom_{M^G(S)}(K_0(G,-)^n,K_k(G,-))  \]
		for each $n,k\in\mathbb{N}$. In particular we get,
		\[ [\ZZ\times Gr_G,\ZZ\times Gr_G]_{\mc{H}^G(S)}\iso End_{M^G(S)}(K_0(G,-)) \]
		when $n=1$ and $k=0$.
	\end{theorem}

Here $\ZZ\times Gr_G$ is pointed at $(0,0)$ (see Remark \ref{Gr(F)}). First note that, as in the non-equivariant case, there is a functor $\mc{H}^G(S)\to M^G(S)$ given by $E\mapsto [-,E]_{\mc{H}^G(S)}$ 
	where $[-,E]_{\mc{H}^G(S)}$ is well defined as an object in $M^G(S)$. Let us denote this functor by $\underline{\pi}_0$. We then have a map of sets
	\begin{equation}\label{eq:alpha}
	\alpha_{F,E}: [F,E]_{\mc{H}^G(S)}\to Hom_{M^G(S)}(\underline{\pi}_0F,\underline{\pi}_0E)
	\end{equation}
	for any pair of objects $E,F\in \mc{H}^G(S)$. We have $\underline{\pi}_0\mathbf{R}\Omega^k K(G,-)\cong K_k(G,-)$ in $M^G(S)$. As $\ZZ\times Gr_G$ represents equivariant algebraic $K$-theory in $\mc{H}^G(S)$, the map \ref{eq:alpha} becomes
	\[ [(\ZZ\times Gr_G)^n,\mathbf{R}\Omega^k\ZZ\times Gr_G]_{\mc{H}^G(S)}\xrightarrow{\alpha} Hom_{M^G(S)}(K_0(G,-)^n,K_k(G,-))\]
	when applied to the pair $((\ZZ\times Gr_G)^n,\mathbf{R}\Omega^k\ZZ\times Gr_G)$. To prove Theorem~\ref{thm:main} we just need to show that this map is a bijection. 
	We will do this by proving some general results below along the lines of \cite{R10}. First we have the equivariant analogue of \cite[Lemma~1.2.1]{R10}. 
	\begin{lemma}\label{lem:lim1}
		 Let $E$ be a group object in $\mc{H}_\bullet^G(S)$. Let $(X_i)_{i\in I}$ be a directed system of objects in $Sm^G_S$ ordered by the directed set $I$ which contains $\mathbb{N}$ as a cofinal subset. Let $X=\colim_{i\in I} X_i$ be the colimit as objects in $M^G(S)$. Then, there exists a short exact sequence of groups,
		\[0\to {\lim_{i\in I}}^1 [X_i,\mathbf{R}\Omega^1 E]_{\mc{H}^G(S)}\to [\colim_{i\in I} X_i, E]_{\mc{H}^G(S)}\to \lim_{i\in I} [X_i, E]_{\mc{H}^G(S)}\to 0 .\]
		In particular, there is an isomorphism of groups
		\[ [\colim_{i\in I} X_i, E]_{\mc{H}^G(S)}\iso \lim_{i\in I} [X_i, E]_{\mc{H}^G(S)} \]
		whenever $([X_i,\mathbf{R}\Omega^1E]_{\mc{H}^G(S)})_{i\in I}$ is a Mittag-Leffler system.
	\end{lemma}
	\begin{proof}
		Let $E_f\in M_\bullet^G(S)$ be a fibrant model of $E$ with respect to the equivariant motivic model structure. We then have for any $F\in M^G(S)$ a bijection,
		\[ [F, E]_{\mc{H}_\bullet^G(S)}\cong [F, E_f]_{\mc{H}_\bullet^G(S)}\cong Hom_{M_\bullet^G(S)}(F, E_f)/\sim \]
		where $\sim$ denotes the equivalence relation induced by simplicial homotopy. For any $Y\in Sm^G_S$ we then have isomorphisms of groups,
		\[ [Y, \mathbf{R}\Omega^n E]_{\mc{H}^G(S)}\cong [Y_+, \mathbf{R}\Omega^n E]_{\mc{H}_\bullet^G(S)} \cong [S^n\wedge Y_+,E]_{\mc{H}_\bullet^G(S)}\cong [S^n\wedge Y_+,E_f]_{\mc{H}_\bullet^G(S)}\cong \pi_n E_f(Y)\]  
		for all $n\geq 0$. Applying this to the $X_i$'s we get $[X_i, \mathbf{R}\Omega^n E]_{\mc{H}^G(S)}\cong \pi_nE_f(X_i)$ for each $i,n\geq 0$. Also note that $\mathbf{Hom}_{M^G(S)}(\colim_{i\in I} X_i,E_f)\cong \lim_{i\in I}E_f(X_i)$. Using these two results we have reduced the problem to finding a short exact sequence
		\[0\to {\lim_{i\in I}}^1\pi_1 E_f(X_i)\to \pi_0 (\lim_{i\in I}E_f(X_i))\to \lim_{i\in I}\pi_0(E_f(X_i))\to 0\]
		 which is a classical result\cite[Section VI.2]{GJ99}.
		\end{proof}
		We will only be concerned with the cases where the $\lim^1$ term vanishes so we give this case a separate name along the lines of \cite{R10}.
		\begin{definition}\label{def:phantom}
			Let $E$ and $(X_i)_{i\in I}$ be as in Lemma \ref{lem:lim1}. We will say that the direct system $(X_i)_{i\in I}$ does not unveil phantoms in $E$ if $ ([X_i,\mathbf{R}\Omega^n E]_{\mc{H}^G(S)})_{i\in I}$ is a Mittag-Leffler system for all $n$. 
		\end{definition}
 		For simplicity we fix a functorial factorisation $E\xrightarrow{\eta_E} E_f\to *$ which gives us a fibrant replacement functor $E\mapsto E_f$ and a natural transformation $\eta:1_{M_\bullet^G(S)}\Rightarrow (-)_f$ which is a weak equivalence at each object. For each $E\in M_\bullet^G(S)$, $\eta_E$ induces a morphism $E\to\underline{\pi}_0E $ in $M_\bullet^G(S)$ given by 
 		\[ E(X)\xrightarrow{\eta_E} E_f(X)\to [X,E]_{_{\mc{H}_\bullet^G(S)}}=\underline{\pi}_0E(X) \]
 		for all $X\in Sm^G_S$. 
 		Now we need an equivariant version of Jouanolou's trick. Hoyois has given a version of it for $G$-quasi-projective schemes \cite[Proposition~2.20]{H17}. Unlike the non-equivariant case not every $G$-scheme can be covered by $G$-invariant affine opens. However, in the case of constant finite groups schemes every $G$-scheme has an equivariant \emph{Nisnevich} cover comprised of affine schemes \cite[Lemma~2.20]{HKO14}. We generalise the notion of affine bundles accordingly.
 		\begin{definition}
 			An $N$-affine bundle over a $G$-equivariant scheme $X$ is a $G$-equivariant $X$-scheme $Y\to X$ such that $Y\times U_i\to U_i$ is a $G$-vector bundle for some Nisnevich cover of $X$, $\{ U_i\to X \}_{i\in I}$.   
 		\end{definition}
 		It follows from this definition that if $\pi:Y\to X$ is an $N$-affine bundle then $\pi$ is a $G$-motivic weak equivalence. Using this we get a straightforward generalisation of Hoyois's result.
 		\begin{theorem}\label{thm:Jtrick}
 			Let $S$ be a regular scheme and $G\to S$ be a constant finite group scheme over $S$ with $\textsl{g}^{-1}\in \Gamma(S,\mc{O}_S)$. For every $X\to S \in Sm^G_S$, there exists an $N$-affine bundle $W\to X$ where $W$ is (globally) affine.
 		\end{theorem}
 		We now have sufficient results to prove the following.
 		\begin{theorem}\label{thm:gen_bijection}
 		 Let $E$ be a group object in $\mc{H}^G(S)$. Let $(X_i)_{i\in I}$ be a directed system of objects in $Sm^G_S$ which does not unveil phantoms in $E$. Let $X=\colim_{i\in I} X_i$ be the colimit in $M^G(S)$. If $X(U)\to \underline{\pi}_0X(U)$ is onto for all affine $U\in Sm^G_S$ then the map
 			\[ [X,E]_{\mc{H}^G(S)}\to Hom_{M^G(S)}(\underline{\pi}_0 X, \underline{\pi}_0E) \] 
 			as given in $(\ref{eq:alpha})$ is a bijection.
 		\end{theorem}
 		\begin{proof}
 			As $X$ does not unveil phantoms in $E$, we have an isomorphism of groups (and hence a bijection) $[X,E]_{\mc{H}^G(S)} \cong \lim_{i\in I} [X_i,E]$. This fits into a commutative diagram,
 			\[
 				\begin{tikzcd}
 				{[X,E]}_{\mc{H}^G(S)}\arrow[r,"\alpha_{X,E}"]\arrow[rd,"\sim"] &Hom_{M^G(S)}(\underline{\pi}_0 X, \underline{\pi}_0E)\arrow[d,"\beta"] \\
 				&\lim_{i\in I} [X_i,E]_{\mc{H}^G(S)}
 				\end{tikzcd}
 				\]
 				where $\beta$ is the composite map 
 				\[Hom_{M^G(S)}(\underline{\pi}_0 X, \underline{\pi}_0E)\to Hom_{M^G(S)}(X, \underline{\pi}_0E)\cong \lim_{i\in I} [X_i,E]_{\mc{H}^G(S)}.\]
 				It follows that $\alpha_{X,E}$ is an injection and $\beta$ is a surjection. To show $\alpha_{X,E}$ is a bijection it is enough to show that $\beta$ is an injection. Given any pair $f_0,f_1\in Hom_{M^G(S)}(\underline{\pi}_0 X, \underline{\pi}_0E)$ which are mapped onto the same element in $Hom_{M^G(S)}(X, \underline{\pi}_0E)$, we have the diagram
 				\[ X\to\underline{\pi}_0 X\xrightarrow[f_1]{f_0} \underline{\pi}_0E \]
 				where $X(U)\to \underline{\pi}_0 X(U)$ is a surjection for all $U$ affine.  Hence $f_0$ and $f_1$ are equal when evaluated on affine $G$-schemes. As $\underline{\pi}_0 X$ and $\underline{\pi}_0E$ are invariant under equivariant motivic equivalences, Theorem \ref{thm:Jtrick} implies that $f_0=f_1$ and we are done.   
 		\end{proof}
 	
 \section{\texorpdfstring{$K$}{K}-theory of Grassmannians} \label{Sec:K-theory_of_grassmannians}	
	For any $X\in Sm^G_S$, there is a canonical map of sets
	\begin{equation}\label{eq:maptopi0}
	Hom_{M^G(S)}(X,\ZZ\times Gr_G)\to K_0(G,X) 
	\end{equation}
	induced by the map in Theorem~\ref{thm:MT}. The infinite Grassmannian $\ZZ\times Gr_G$ is given by the filtered colimit $\colim_{i,r,n} \{i \}\times  Gr(r\textsl{g},\rho^{\oplus n})$, with $r,n\in \mathbb{N}$ and $i\in\ZZ$. Hence we have,
    \[\colim_{i,r,n} Hom_{M^G(S)}(X,\{i\}\times Gr(r\textsl{g},\rho^{\oplus n})) \iso Hom_{M^G(S)}(X,\ZZ\times Gr_G)\]
    and \ref{eq:maptopi0} is given by 
    \[ (i,M\mono \rho^{\oplus n})\mapsto [M]-r[\rho]+i[\mathcal{O}_X] \]
    where $\mathcal{O}_X$ has the trivial action. Here we use the fact that the morphism $Gr(r\textsl{g},\rho^{\oplus n})\to Gr(r+k\textsl{g},\rho^{\oplus n+k+m})$  is given by $M\mono \rho^{\oplus n}\mapsto (M\oplus\rho^{\oplus k}\mono \rho^{\oplus n}\oplus\rho^{\oplus k+m}$).
	\begin{theorem}\label{th:ontopi0}
		For all affine $X\in Sm^G_S$, the map in $\mathrm(\ref{eq:maptopi0})$ is onto.
	\end{theorem}
	\begin{proof}
     It is sufficient to show this for $G$-connected schemes. Let $X\cong \spec R$ be affine and $G$-connected. There is an epimorphism of $G$-vector bundles $s:\rho \to R$ given by $s(\sum_g r_ge_g)=\sum r_g$. Let $I$ denote the kernel of this map. The morphism $Hom_{M^G(S)}(X,\{i\}\times Gr(r\textsl{g},\rho^{\oplus n}))\to K_0(G,X)$ is given by
		\begin{equation}\label{loc:map}
		(i,M\mono \rho^{\oplus n})\mapsto [M]-r[\rho]+i[R] 
		\end{equation}  
		 for each $r,n\in\mathbb{N}$ and $i\in\ZZ$ as discussed above. Let $[P]-[Q]\in K_0(G,\spec R)$. By Lemma~\ref{Equivalent Category of Summands in Regular Representations}, $Q$ is a direct summand of $\rho^{\oplus n}$ for some $n$. Therefore $[Q]=n[\rho]-[Q']$ for some $G$-vector bundle $Q'$ and we have 
		 \[[P]-[Q]=[P]-n[\rho]+[Q']=[P\oplus Q']-n[\rho]=[P']-n[\rho]\]
		 in $K_0(G, \spec R)$. If rank of $P'$ is $p=b\textsl{g}+r$ with $r<\textsl{g}$ then we can add and subtract $(\textsl{g}-r)$ copies of $[R]$ to get 
		 \[ [P']-n[\rho]=[P'\oplus R^{\oplus \textsl{g}-r}]-n[\rho]-(\textsl{g}-r)[R] \]
		 where $P'\oplus R^{\oplus \textsl{g}-r}$ then has rank a multiple of $\textsl{g}$. Therefore every element in $K_0(G,\spec R)$ is of the form $[M]-n[\rho]-i[R]$ where rank of $M$ is a multiple of $\textsl{g}$. By \ref{loc:map} it is enough to show that any element of this form is equal to some $[M']-n'[\rho]-i'[R]$ with rank of $M=n'\textsl{g}$. Let $rank(M)=m\textsl{g}$ we then have two cases, $m<n$ and $m>n$. If $m<n$ then 
		 \begin{multline*}
		 [M]-n[\rho]-i[R]=[M]+(n-m)\textsl{g}[R]-(n-m)\textsl{g}[R]-n[\rho]-i[R]\\=[M\oplus R^{\oplus (n-m)\textsl{g}}]-n[\rho]-(n\textsl{g}-m\textsl{g}+i)[R]
		 \end{multline*}
		 which is in the desired form. When $m>n$ then 
		 \begin{multline*}
		 	[M]-n[\rho]-i[R]=[M]+(m-n)\textsl{g}[I]-(m-n)\textsl{g}[I]-n[\rho]-i[R]\\
		 	=[M\oplus I^{\oplus (m-n)\textsl{g}}]-(m-n)\textsl{g}[\rho]-n[\rho]+((m-n)\textsl{g}-i)[R]\\=[M\oplus I^{\oplus (m-n)\textsl{g}}]-((m-n)\textsl{g}+n)[\rho]+((m-n)\textsl{g}-i)[R]
		 \end{multline*}  
		 which is of the desired form.
	\end{proof}
	To prove Theorem~\ref{thm:main} we need to calculate the $K$-groups of the Grassmannian schemes. We have from \cite[Theorem~3.1]{T87} the $G$-equivariant projective bundle theorem. 
	\begin{theorem}[Projective bundle theorem]\label{thm:pbt}
		Let $X$ be a $G$-equivariant $S$-scheme and let $\mc{E}$ be a $G$-vector  bundle over $X$ of constant rank $n$. There is a canonical $G$-action on $\mathbb{P}(\mc{E})$ induced by the action on $E$. The functor
		 \[\prod^n_{i=1} VB(G,X)\to VB(G,\mathbb{P}(\mc{E})) \]
		 given by $(E_1,E_2,\dots,E_n)\mapsto \oplus^n_{i=1}f^*E_i\otimes\mc{O}(-i)$
		 induces a homotopy equivalence $K(G, X)^n\iso K(G, \mathbb{P}(\mc{E}))$.
	\end{theorem}
	We can go further and describe $K_*(G,\mathbb{P}(\mc{E}))$ as a $K_*(G, X)$-algebra along the same lines as the non-equivariant case.
	\begin{corollary}\label{cor:pbtK0}
		Let $X$ and $\mc{E}$ be as above. The homotopy equivalence $K(G,X)^n\iso K(G,\mathbb{P}(\mc{E}))$ given in Theorem $\ref{thm:pbt}$ induces an isomorphism
		\[ K_*(G,\mathbb{P}(\mc{E}))\cong K_*(G, X)\otimes_{K_0(G, X)}K_0(G, \mathbb{P}(\mc{E})) \]
		of $K_*(G,X)$-algebras. 
	\end{corollary}
	\begin{proof}
		First, by Theorem \ref{thm:pbt} we have a commutative diagram of $K_0(G, X)$-modules 
		\[
			\begin{tikzcd}
		 		K_0(G, X)^n\arrow[r]\arrow[d,"\sim" allign] & K_*(G, X,)^n \arrow[d,"\sim" allign]\\
				K_0(G, \mathbb{P}(\mc{E}))\arrow[r] & K_*(G, \mathbb{P}(\mc{E}))
                \end{tikzcd}
			\]
		where the vertical maps are isomorphisms and the bottom map is a ring homomorphism. Using the extension-restriction adjunction with respect to the ring homomorphism $K_0(G, X)\to K_*(G, X)$ we get the following commutative diagram 
		\[
		 	\begin{tikzcd}
		 		K_*(G, X)\otimes_{K_0(G, X)} K_0(G, X)^n\arrow[r]\arrow[d,"\sim" allign] & K_*(G, X)^n \arrow[d,"\sim" allign]\\
		 		K_*(G, X)\otimes_{K_0(G, X)}K_0(G, \mathbb{P}(\mc{E}))\arrow[r] & K_*(G, \mathbb{P}(\mc{E}))
		 		\end{tikzcd}
		 	\]
		 of $K_*(G, X)$-modules. Here again, the bottom map is a ring homomorphism. Note that the top map is bijection and hence so is the bottom map.
		\end{proof}
	 We can give an explicit description of the ring $K_0(G, \mathbb{P}(\mc{E}))$ using $\lambda$-operations. For any $X$ let the $n^{th}$ lambda operation map $\lambda^n:K_0(G, X)\to K_0(G, X)$ be given by the exterior powers, $[\mc{E}]\mapsto [\wedge^n\mc{E}]$. The collection $(\lambda^n)_{n\in\mathbb{N}}$ defines a $\lambda$-ring structure on $K_0(G,X)$. The details are given in \cite[Chapter II]{W13} for the non-equivariant case and can be straightforwardly extended to the equivariant case.
	\begin{corollary}
		Let $X$ be a $G$-equivariant $S$-scheme and let $\mc{E}$ be a $G$-vector bundle over $X$ of constant rank $n$. Let $f:\mathbb{P}(\mc{E})\to X$ be the corresponding projective bundle. The $K_0(G,X)$-algebra homomorphism $c:K_0(G, X)[t]\to K_0(G, \mathbb{P}(\mc{E}))$ given by $t\mapsto [\mc{O}(-1)]$ induces an isomorphism of rings,
		\begin{align*}
			& K_0(G,X)[t]/(\theta) \cong K_0(G,\mathbb{P}(\mc{E})),	&\mathrm{where}\:\: \theta= \sum_{i=0}^{n-1} (-1)^i \lambda^{n-i} [\mc{E}]t^i	
			\end{align*}	 
			and hence by Corollary \ref{cor:pbtK0} we get an isomorphism of graded $K_*(G, X)$-algebras, 
			\[K_*(G, X)[t]/(\theta) \cong K_*(G, \mathbb{P}(\mc{E})).\]
		\end{corollary}	
	\begin{proof}
		Let $c(\mc{E})=[f^*\mc{E}]$ and $t=[\mc{O}(-1)]$ in $K_0(G, \mathbb{P}(\mc{E}))$. Consider the element $(c(\mc{E})-t)\in K_0(G, \mathbb{P}(\mc{E}))$. We have $(c(\mc{E})-t) = [\mc{F}]$ where $\mc{F}$ is the quotient,
		\[ \mc{O}(-1)\mono f^*\mc{E}\epi \mc{F} \]
		over $\mathbb{P}(\mc{E})$ and hence $\lambda^n(c(\mc{E})-t)= [(\wedge^n\mc{F})]=0$. Using the properties of $\lambda$-structures we get 
		\[ 0=\lambda^n(c(\mc{E})-t)=\sum_{i=0}^{n}\lambda^i(c(\mc{E}))\lambda^{n-i}(-t)\] 
		in $K_0(G, \mathbb{P}(\mc{E}))$. As $t$ is a line bundle, we have $\lambda^1(t)=t$ and $\lambda^i(t)=0$ for $i\geq 1$. By the properties of $\lambda$-operations we have $\lambda^{i}(-t)=(-1)^{i}t^i$ for all $i\geq 0$. Putting this into the equation above we get 
		\[ t^n=\sum_{i=1}^{n}\lambda^i(c(\mc{E}))(-1)^{n-i}t^{n-i} \]
		 and by Theorem~\ref{thm:pbt} we know that $K_0(G,\mathbb{P}(\mc{E}))$ is a free module over the ring $K_0(G,X)$ with basis $ \{t,\ldots,t^{n-1} \} $. Hence, we are done.
		\end{proof}
	Just like in the non-equivariant case, we can use this to calculate $K_n(G, Gr(r,\mc{E}))$ for all Grassmannians. Let $\mathcal{E}$ be a $G$-vector bundle of rank $n$ over $X$. For every partition $\pi=(p_1,p_2,\ldots,p_r)$ of $n$, the flag bundle $f:Flag(\pi,\mathcal{E})\to X$ has an induced $G$-action. As a moduli space, given $t:Y\to X$ and $g\in G$, the action is given by 
    \[ (U_1\into\ldots\into U_r\cong t^*\mathcal{E}) \mapsto ( g^*U_1\into\ldots\into g^*U_r\cong g^*t^*\mathcal{E}\cong t^*\mathcal{E}) \] where $rk(F_i)=\sum_{j\leq i} p_j$ and the final isomorphism is induced by the action on $\mathcal{E}$. If $Y$ has a $G$-action then equivariant morphisms $Y\to Flag(\pi, \mathcal{E})$ correspond to equivariant flags, that is, sequences $G$-subbundles. Therefore, the canonical flag
	\[ F_1\into F_2\into\ldots\into F_r\cong f^*\mathcal{E} \]
	is equivariant. When the partition is of the form $\pi=(1,\ldots,1,n-r)$, we denote this bundle by $Flag(1^r,\mathcal{E})$. These bundles are in fact iterated projective bundles, 
	\[\mathbb{P}(\mathcal{E})\leftarrow Flag((1,1,n-1),\mathcal{E})\leftarrow \ldots \leftarrow Flag((1^r,n-1),\mathcal{E}) \]
	over $X$. We then get the following corollary.
	\begin{corollary}\label{cor:K_completeflag}
		Let $X$ be a $G$-equivariant $S$ scheme, and let $\mc{E}$ be a $G$-vector bundle over $X$ of constant rank $n$. Let $Flag(1^r,\mathcal{E})$ denote the corresponding flag bundle with canonical flag
		\[ F_1\into F_2\into\ldots\into F_r\cong f^*\mathcal{E} \]
		where $f:Flag(1^r,\mathcal{E})\to X$ is the structure map. The K-theory ring $K_*(G, Flag(1^r,\mathcal{E}))$ is generated as a $K_*(G, X)$-algebra by the elements $t_i=[F_{i}]-[F_{i-1}]$ in $K_0(G, Flag(1^r,\mathcal{E}))$. Furthermore, $K_*(G, Flag(1^r,\mathcal{E}))$ is a free $K_*(G, X)$-module with basis,
		\[\{t_1^{a_1}t_2^{a_2}\ldots t_r^{a_r} |0\leq a_i\leq n-i, i\leq r  \} \]
		and hence of rank $\frac{n!}{n-r!}$. 	
		\end{corollary} 
		We can use this result to compute the $K$-theory ring of any Grassmannian. 
	\begin{theorem}\label{thm:K_Gr}
		Let $f:Gr(r,\mathcal{E})\to X$ be a Grassmannian over any scheme $X$ with canonical $G$-subbundle $\mathcal{U}_r\into f^*\mathcal{E}$. The $K_*(G,X)$-algebra  $K_*(G,Gr(r,\mathcal{E}))$ is faithfully flat and is generated by $\lambda^i=[\Lambda^i\mathcal{U}_r]$, for $1\leq i\leq r$, in $K_0(G,Gr(r,\mathcal{E}))$. 
		\end{theorem}
	\begin{proof}
		Given a Grassmannian $Gr(r,\mathcal{E})$, the projective bundle $\mathbb{P}(\mathcal{U}_r)\to Gr(r,\mathcal{E})\to X$ is the flag bundle $Flag((1,r,n-r-1),\mathcal{E})\to X$. We can then recursively construct a sequence 
		\[X\leftarrow Gr(r,\mathcal{E})\leftarrow Flag((1,r,n-r-1),\mathcal{E})\leftarrow Flag((1,1,r,n-r-2),\mathcal{E})\leftarrow\ldots Flag((1^r,\mathcal{E})) \]
		of flag bundles. Let the canonical flag over $Flag(1^r,\mathcal{E})$ be given by
		\[ F_1\into F_2\into\ldots\into F_r\into f^*\mathcal{E}\]
		 with $F_r=\mathcal{U}_r$. By Corollary \ref{cor:K_completeflag} the morphism $K_*(G,Gr(r,\mathcal{E}))\to K_*(G,Flag(1^r,\mathcal{E}))$ is an inclusion and $K_*(G,Flag(1^r,\mathcal{E}))$ is generated as a $K_*(G,Gr(r,\mathcal{E}))$-algebra by $t_i=[F_i]-[F_{i-1}]$ for $0<i\leq r$. In addition, the image of $\lambda^i=[\Lambda^i\mathcal{U}_r]$ in $K_*(G,Flag(1^r,\mathcal{E}))$ is the $i^{th}$ elementary symmetric polynomial $e_i(t_1,\ldots,t_r)$. We therefore get a morphism $K_*(G,X)[e_1,\ldots,e_r]\to K_*(G,Gr(r,\mathcal{E}))$ such that the following diagram commutes
		 \[      \begin{tikzcd}
		 		K_*(G, X)[e_1,\ldots,e_r]\arrow[r]\arrow[d] &K_*(G, X)[t_1,\ldots,t_r]\arrow[d]\\
		 		K_*(G,Gr(r,\mathcal{E}))\arrow[r] &K_*(G,Flag(1^r,\mathcal{E}))
		 		\end{tikzcd}
		 	\] 
		 	where the right vertical map is a quotient map. For any ring $R$ the polynomial ring $R[t_1,\ldots,t_r]$ is a free module over the ring of symmetric polynomials $R[e_1,\ldots,e_r]$ with basis 
		 	\[\{t_1^{a_1}t_2^{a_2}\ldots t_r^{a_r} |0\leq a_i\leq r-i, i\leq r  \}\]
		 	which is precisely the basis of $K_*(G,Flag(1^r,\mathcal{E}))$ as a $K_*(G, Gr(r,\mathcal{E}))$-module. Hence the left vertical map is a surjection. Faithfully flatness follows from the two-out-of-three property.
		\end{proof}
		
		\section{Proof of theorem \ref{thm:main}}
	We now have enough to prove Theorem \ref{thm:main}. For each $i\in \mathbb{N}$, let $[-i,i]$ denote the set of all natural numbers $n$ that satisfy $-i\leq n \leq i$. Then for each $k\in \mathbb{N}$, $\ZZ\times Gr_G^k$ is the directed colimit, $\colim_{i\in N} [-i,i]\times Gr_G(i,\rho^{2i})^k$. We will need the following lemma.
	\begin{lemma}\label{lem:K_Mittag-Leffler}
		For each $k$, the directed system $(([-i,i]\times Gr(i,\rho^{2i}))^k)_{i\in \mathbb{N}}$ does not unveil phantoms in $K(G,\ -)$.
		\end{lemma}
	\begin{proof}
		By Theorem \ref{thm:MT} we have a bijection 
		\[ [X,\mathbf{R}\Omega^n\ZZ\times Gr_G]_{\mc{H}^G(S)}\xrightarrow{\alpha} K_n(G, X)\]
		for all schemes $X$ and all $n\geq 0$. By Lemma \ref{lem:lim1} it is therefore enough to prove that $K_n(G,(([-i,i]\times Gr(i,\rho^{2i}))^k)_{i\in \mathbb{N}})$ is Mittag-Leffler for all $n,k\geq 0$. We will achieve this by showing that the morphism of graded rings 
		\[K_*(G,([-i-1,i+1]\times Gr((i+1)\textsl{g},\rho^{2(i+1)}))^k)\to K_*(G,([-i,i]\times Gr(i\textsl{g},\rho^{2i}))^k) \]
		is surjective for all $i\geq 0$. As $K_*(G,X\coprod Y)\iso K_*(G,X)\oplus K_*(G,Y)$  it is enough to prove surjectivity of 
        \[K_*(G, Gr((i+1)\textsl{g},\rho^{2(i+1)})^k)\to K_*(G, Gr(i\textsl{g},\rho^{2i})^k) \]
        for all $k\geq 0$.
        For $k=1$ the induced ring homomorphism 
		\[f_i:K_*(G,Gr((i+1)\textsl{g},\rho^{2(i+1)}))\to K_*(G,Gr(i\textsl{g},\rho^{2i})) \] 
		sends $[\mathcal{U}_{\textsl{g}(i+1)}]-[\rho]$ to $[\mathcal{U}_{\textsl{g}i}]$ in $K_0(G,Gr(i\textsl{g},\rho^{2i}))$. The homomorphism restricted to $K_0(G,-)$ preserves the $\lambda$-ring structure and hence $f_i$ is surjective for all $i$. For the general case note that as a consequence of Theorem \ref{thm:K_Gr} for any $f:Y\to X$ and $Gr(r,\mc{E})\to X$ we have
        \[ K_*(G,Gr(r,f^*\mc{E}))\cong K_*(G,Y\times_X Gr(r,\mc{E}))\cong K_*(G,Y)\otimes_{K_*(G, X)} K_*(G, Gr(r,\mc{E})) \]
        and applying this to
        $Gr(i\textsl{g},\rho^{n})\times Gr(j\textsl{g},\rho^{m})$ we get \[K_*(G,Gr(i\textsl{g},\rho^{n}))\otimes_{K_*(G,X)} K_*(G,Gr(j\textsl{g},\rho^{m}))\cong K_*(G,Gr(i\textsl{g},\rho^{n})\times Gr(j\textsl{g},\rho^{m})).\]
		Hence $\bigotimes^k_{j=0} K_*(G,Gr(i\textsl{g},\rho^{2i}))\iso K_*(G,Gr(i\textsl{g},\rho^{2i})^k)$ for all $k$ where the tensor product on the left is as $K_*(G,X)$-algebras. Combined with the $k=1$ case this shows that the desired map is surjective.
		\end{proof}
		We now have enough to prove Theorem~\ref{thm:main}.
	\begin{proof}[Proof of Theorem $\ref{thm:main}$]
		By Lemma~\ref{lem:K_Mittag-Leffler} and Theorem~\ref{th:ontopi0}, we can apply Theorem~\ref{thm:gen_bijection} to $\ZZ\times Gr_G=\colim_{i\in \mathbb{N}} [-i,i]\times Gr(i\textsl{g},\rho^{2i})$, giving us the desired result.
		\end{proof}
	As in the non-equivariant case, the above bijection respects algebraic structures hence we get the equivariant analogue of \cite[Theorem~2.3.1]{R10}.
	\begin{corollary}
	\label{Algebraic Structures}
		In the category $\mc{H}^G(S)$ there exists a unique structure of a special $\lambda$-ring with duality on $\ZZ\times Gr_G$ such that the corresponding $\lambda$-ring with duality structure on $K_0(G,X)$ is the standard one for all $X\in Sm^G_S$.
		\end{corollary}

\end{document}